\documentclass{amsart}
\usepackage[T1]{fontenc}

\usepackage{esvect}
\usepackage{amssymb, mathrsfs}
\usepackage{enumerate}
\usepackage[colorinlistoftodos]{todonotes}
\usepackage{verbatim}

\usepackage{amsthm, url}
\usepackage{xcolor}
\usepackage{tikz,pgf}
\usepackage{tikz-cd}
\usetikzlibrary{decorations.text}

\newtheorem{theorem}{Theorem}[section]

\newtheorem{corollary}[theorem]{Corollary}
\newtheorem{proposition}[theorem]{Proposition}
\theoremstyle{definition}
\newtheorem{definition}[theorem]{Definition}

\newtheorem{remark}[theorem]{Remark}
\numberwithin{equation}{section}

\newcommand{\U}[1]{{\mathcal{U}_{#1}}}
\newcommand{\V}[1]{{\mathcal{V}_{#1}}}
\newcommand{\W}[1]{{\mathcal{W}_{#1}}}
\newcommand{\sE}{\mathcal{E}}
\newcommand{\sF}{\mathcal{F}}
\newcommand{\df}[1]{{{\bf #1}}}

\renewcommand{\epsilon}{\varepsilon}

\begin{document}


\renewcommand{\bf}{\bfseries}
\renewcommand{\sc}{\scshape}

\title[Coarse direct products and property {C}]{Coarse direct products and property {C}}
\author{G.~Bell}
\address{Department of Mathematics and Statistics, The University of North Carolina at Greensboro, Greensboro, NC 27402, USA} 
\email{gcbell@uncg.edu}

\author{A.~Lawson}
\address{Department of Mathematics and Statistics, The University of North Carolina at Greensboro, Greensboro, NC 27402, USA} 
\email{azlawson@uncg.edu}
\date{\today}

\begin{abstract}
We show that coarse property~C is preserved by finite coarse direct products. We also show that the coarse analog of Dydak's countable asymptotic dimension is equivalent to the coarse version of straight finite decomposition complexity and is therefore preserved by direct products. 
\end{abstract}

\maketitle
\section{Introduction and Preliminaries}

The coarse category was described by Roe~\cite{roe2003} as a generalization of the large-scale approach to discrete groups begun by Gromov~\cite{gromov93}. Coarse spaces are sets that are equipped with a so-called coarse structure that provides a measure of proximity without referring to a metric. Coarse structures can be derived from metric structures~\cite{roe2003,wright-dissertation}, topological structures~\cite{roe2003}, or group structures~\cite{nicas-rosenthal}. Coarse versions of asymptotic dimension~\cite{grave2005,roe2003} as well as property~C and finite decomposition complexity~\cite{bell-moran-nagorko2016} have been established and studied~\cite{Yamauchi2017}.

There is a natural notion of a direct product of coarse spaces whose structure is not necessarily derived from a coarse structure on the product of the underlying sets, see Remark~\ref{remark}. Therefore, the question of stability of coarse properties under this product is interesting. The primary goal of this short note is to show that coarse property~C is stable with respect to finite coarse products; this was shown in the metric case recently~\cite{bell-nagorko2016,Davila}. We also show that the coarse analog of Dydak's countable asymptotic dimension~\cite{dydak2016} coincides with the coarse version of straight finite decomposition complexity (sFCDC); as a result, this notion is also stable with respect to coarse products. 

A coarse structure can be defined on any set $X$. Take the multiplication (referred to here as composition) and inverse operations from the pair groupoid structure on the product $X\times X$. A collection $\mathcal{E}$ of subsets of $X\times X$ is called a \df{coarse structure} on $X$ if it contains the diagonal and is closed under subsets, finite unions, inverses, and compositions~\cite{roe2003}. 
Elements of $\sE$ will be called \df{entourages} and we call the pair $(X,\sE)$ a \df{coarse space}. 

Suppose $(X_1,\sE_1)$ and $(X_2,\sE_2)$ are coarse spaces. Let $p_i$ denote the projection map $p_i:X_1\times X_2\to X_i$ for $i\in\{1,2\}$.  We define the \textbf{product coarse structure}~\cite{grave2005} on the product $X_1\times X_2$ by 
\[\sE_1\ast\sE_2=\left\{E\subseteq(X_1\times X_2)^2\colon(p_i\times p_i)(E)\in\sE_i\hbox{ for each } i\in\{1,2\}\right\}\hbox{.}\]

\begin{remark}\label{remark}
The $C_0$-coarse structure on metric spaces was defined by Wright~\cite{wright-dissertation}. 
It is straightforward to check that the $C_0$-coarse structure on the metric product $\mathbb{N}\times [0,1]$ is a proper subset of the coarse product of the $C_0$-coarse structures on the individual spaces. Therefore, a coarse structure on the product of two spaces does not necessarily correspond to the coarse product of those same types of coarse structure on the factors. By contrast, if $X$ and $Y$ are metric spaces in the bounded coarse structure~\cite[Example 2.5]{roe2003}, the coarse space obtained from taking the bounded coarse structure on $X\times Y$ is the same as the coarse product of the spaces $X$ and $Y$ equipped with bounded coarse structures.
\end{remark}

To complete this section, we recall the definitions of coarse property C and sFCDC~\cite{bell-moran-nagorko2016}.

\begin{definition}\cite{bell-moran-nagorko2016}
A coarse space $(X,\mathcal{E})$ has \textbf{coarse property C} if and only if for any sequence $E_1\subseteq E_2\subseteq\cdots$ of entourages there is a finite sequence $\U{1}, \U{2},\ldots,\U{n}$ satisfying 
\begin{enumerate}
\item $\bigcup_{i=1}^n \U{i}$ is a cover of $X$;
\item each $\U{i}$ is \textbf{$E_i$-disjoint}; i.e., for any pair of distinct elements $U,V$ of $\U{i}$ we have $U\times V\cap E = \emptyset$; and
\item each $\U{i}$ is \textbf{uniformly bounded}; i.e., $\bigcup_{U\in\U{i}}U\times U\in\mathcal{E}$.
\end{enumerate}
\end{definition}

Let $Y$ be a subset of a coarse space $(X,\sE)$ and let $\U{}$ be a family of subsets of $X$. Let $n$ be a positive integer and let $E\in \sE$ be an entourage. We say that $Y$ admits an $(E,n)$-\df{decomposition} over $\U{}$ if $Y$ can be expressed as a union of $n$ sets $Y^1,Y^2,\ldots,Y^n$ in such a way that each $Y^i$ can be expressed as an $E$-disjoint union of sets from $\U{}$. Here, by an $E$-\df{disjoint union of sets from} $\U{}$, we mean that each $Y^i=\sqcup_j Y^i_j$, where $Y^i_j\times Y^i_{j'}\cap E=\emptyset$ if $j\neq j'$ and $Y^i_j\in\U{}$ for all $j$.

\begin{definition} (\cite{bell-moran-nagorko2016}) The coarse space $(X, \sE)$ is said to have \df{straight finite coarse decomposition complexity} (sFCDC) if for any sequence of entourages $L_1,L_2,...$ there is a sequence of families $\V{1},...,\V{n}$ so that
\begin{enumerate}
\item $\V{1} = \{X\}$;
\item for every $i\ge 1$, each $U\in\V{i}$ admits an $(L_i,2)$-decomposition over $\V{i+1}$; and
\item $\V{n}$ is uniformly bounded.
\end{enumerate}
\end{definition}
It is known that sFCDC is preserved by coarse direct products~\cite[Theorem 4.17]{bell-moran-nagorko2016}.

\section{Results}

The proof that asymptotic property C is preserved by direct products of metric spaces~\cite{bell-nagorko2016,Davila} is based on the technique used to prove the corresponding theorem for topological property C when one of the two factors is compact~\cite{rohm1990}. The proof of the following theorem is similar in spirit to the earlier works, but the absence of a metric means that more care and bookkeeping is required. In particular the single distance parameter must be replaced by two sequences of entourages in the factors.

\begin{theorem}
Let $(X, \sE)$ and $(Y,\sF)$ be coarse spaces with coarse property C. Then $(X\times Y,\sE\ast\sF)$ has coarse property C.
\end{theorem}

\begin{proof}
Let $E_1\subseteq E_2\subseteq\cdots$ be a sequence of entourages in $\sE\ast \sF$. For each $i$, put $K_i=(p_1\times p_1)(E_i)$ and $L_i=(p_2\times p_2)(E_i)$. Then, by the definition of $\sE\ast\sF$, each $K_i\in \sE$ and $L_i\in\sF$. Observe that since $E_i\subseteq E_{i+1}$, we have $K_i\subseteq K_{i+1}$ and $L_i\subseteq L_{i+1}$. Arrange the indices $1,2,3,\ldots$ into a two-dimensional array with the property that the indices are increasing from left to right along rows and from bottom to top along columns (see Figure~\ref{fig:diagram}, which was first used in the metric proof~\cite{bell-nagorko2016}).

For each $i$, we apply the coarse property C definition to the column $K_{i,1}, K_{i,2}, \ldots$ to find an $n_i$ and a cover $\U{i,1},\U{i,2},\ldots,\U{i,n_i}$ of $X$ by families that are uniformly bounded and $K_{i,j}$-disjoint. Then, consider the sequence $L_{1,n_1},L_{2,n_2,},\ldots$. We may assume that the sequence is increasing by replacing $L_{i,n_i}$ by an entourage that occurs higher in the $i$-th column, if necessary. 

Using this sequence, and the fact that $Y$ has coarse property C, we find a cover $\V{1},\V{2},\ldots, \V{m}$ by subsets of $Y$ that are uniformly bounded and so that $\V{i}$ is $L_{i,n_i}$-disjoint.

\begin{figure}
\includegraphics[width=.9\textwidth]{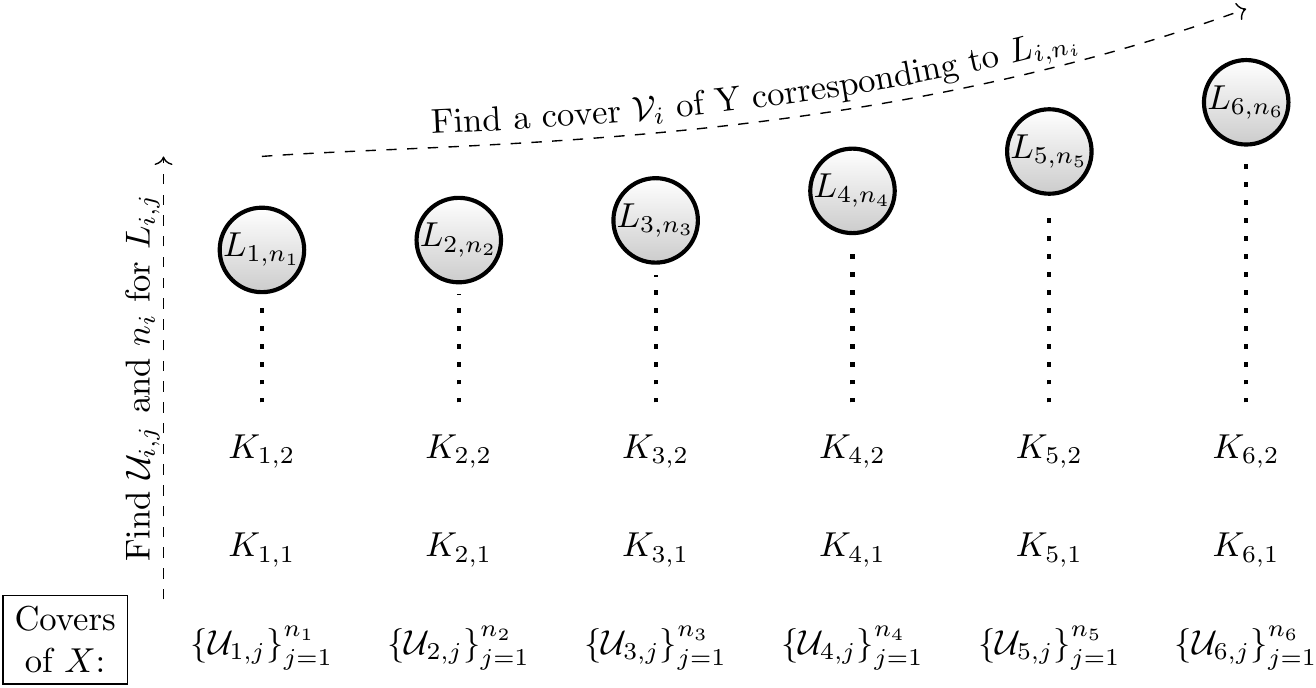}
\caption{We find covers $\{\U{i,j}\}_{j=1}^{n_i}$ of $X$ for each column $K_{i,1},K_{i,2},\ldots,$ and then construct a cover $\{\V{i}\}_{i=1}^m$ of $Y$ corresponding to $L_{i,n_i}$}\label{fig:diagram}
\end{figure}

We claim that the family $\W{i,j}=\{U\times V\colon U\in\U{i,j}, V\in\V{i}\}$ covers $X\times Y$, consists of uniformly bounded sets, and has the property that $\W{i,j}$ is $E_{i,j}$-disjoint. To finish the proof, we simply need to unravel the re-indexing to arrive at the original sequence, which may now include some empty families. 

First we check that the collection $\W{i,j}$ covers $X\times Y$. Given $(x,y)\in X\times Y$, there is some $i_0$ so that $y\in V$, where $V\in \V{i_0}$ since the $\V{i}$ cover $Y$. Next, there is some $j_0$ so that $x\in U$ with $U\in \U{i_0,j_0}$ since for each fixed $i$, the $\V{i,j}$ cover $X$.

To see that the sets are uniformly bounded, we observe that for $\ell=1,2$, \[\left(p_\ell\times p_\ell\right)\left[ \bigcup\left( (U\times V)\times (U\times V)\right)\right]\! =\!\bigcup\left[(p_\ell\times p_\ell)\left( (U\times V)\times (U\times V)\right)\right]\hbox{,}\] where the unions are taken over all $U\times V\in W_{i,j}$ for each pair $i,j$. The conclusion follows from the fact that $\U{i,j}$ and $\V{i}$ are uniformly bounded; i.e., that $\bigcup_{U\in \U{i,j}} U\times U\in \sE$ and $\bigcup_{V\in \V{i}}V\times V\in \sF$.

Finally, we check that $\W{i,j}$ is $E_{i,j}$-disjoint. To this end, take distinct $U_1\times V_1$ and $U_2\times V_2$ in $\W{i,j}$. Assume that there were some $(a,c,b,d)\in L\cap ((U_1\times V_1)\times (U_2\times V_2))$. Then, in particular, $a\in U_1$, $b\in U_2$, $c\in V_1$ and $d\in V_2$. Thus, $(a,b)\in (p_1\times p_1)(E_{i,j})=K_{i,j}$ and $(c,d)\in (p_2\times p_2)(E_{i,j})=L_{i,j}$. Since $U_1\times V_1\neq U_2\times V_2$, we either have $U_1\neq U_2$ or $V_1\neq V_2$. In the first case, the $K_{i,j}$-disjointness of $\U{i,j}$ does not allow $(a,b)$ to be in $K_{i,j}$. In the second case, the fact that $\V{i}$ is $L_{i,n_i}$-disjoint and the fact that $L_{i,j}\subset L_{i,n_i}$ for all $j\le n_i$ means that $(c,d)$ cannot be in $L_{I,j}$. Thus, there can be no such point $(a,c,b,d)$. We conclude that the intersection is empty, which is what we needed to show.
\end{proof}

Dydak defined countable asymptotic dimension for metric spaces~\cite{dydak2016} and it was shown to be equivalent to straight finite decomposition complexity by Dydak and Virk~\cite{dydak-virk2016}. 
We show that the analogous result also holds for coarse spaces. As before, some care is needed to work with entourages in the absence of a metric.

\begin{proposition}\label{prop}
Let $(X,\sE)$ be a coarse space. The following are equivalent:
\begin{enumerate}
    \item there is a sequence $(n_i)$ of integers such that for every sequence of entourages $K_i$ there is a finite sequence of families $\V{1},\ldots,\V{r}$ of subsets of $X$ such that $\V{1}=\{X\}$, every $V\in\V{i}$ admits an $(K_i,n_i)$-decomposition over $\V{i+1}$ and such that $\V{r}$ is uniformly bounded. 
    \item for every sequence $L_i$ of entourages there is a finite sequence of families $\U{1},\ldots \U{s}$ of subsets of $X$ such that $\U{1}=\{X\}$, every $U\in\U{i}$ admits a $(L_i,2)$-decomposition over $\U{i}$ and such that $\U{s}$ is uniformly bounded.
\end{enumerate}
\end{proposition}

\begin{proof}
Clearly (2) implies (1). 

To see the other implication, let $(n_i)$ be the sequence of positive integers satisfying (1) for $X$. Let $L_1, L_2,\ldots$ be a sequence of entourages. By taking unions we may assume $L_i\subseteq L_{i+1}$. Put $K_1=L_{n_1}$, $K_2=L_{n_1+n_2}$, and in general, put $K_j=L_{n_1+\cdots+ n_j}$. Apply (1) with the sequence $(K_i)$ to obtain $\V{1},\V{2},\ldots$ such that $\V{1}=\{X\}$ and such that $\V{i}$ admits a $(K_i,n_i)$-decomposition over $\V{i+1}$.

We will define a sequence $\U{i}$ of families of subsets of $X$ with the property that $\U{1}=\{X\}$ and $\U{i}$ admits an $(L_i,2)$-decomposition over $\U{i+1}$. To begin, we observe that we can write $X=X^1\cup X^2\cup\cdots \cup X^{n_1}$, with each $X^i=\bigsqcup_{\footnotesize K_1-\hbox{disj}}X^i_j,$ and each $X^i_j\in\V{2}$. Therefore, we take $\U{2}=\{X^1_1,X^1_2,\ldots,X^2\cup X^3\cup\cdots\cup X^{n_1}\}.$ Then, it is clear that $X$ can be $(K_1,2)$-decomposed over $\U{2}$ and since $L_1\subseteq L_{n_1}=K_1$, we see that there is an $(L_1,2)$-decomposition of any set in $\U{1}$ over the family $\U{2}$. For $\U{3}$, we take $\{X^1_1,X^1_2,\ldots, X^2_1,X^2_2,\ldots, X^3\cup X^4\cup\cdots \cup X^{n_1}\}$; we also observe that any set in $\U{2}$ admits an $(L_2,2)$-decomposition over $\U{3}$ since $L_2\subseteq L_{n_1}=K_1$.
Continue to define families this way to obtain $\U{1},\U{2},\dots,\U{n_1}$ with an $(L_i,2)$-decomposition of each $\U{i}$ over $\U{i+1}$ for $1\le i<n_1$. We observe that in this way, $\U{n_1}=\V{2}$.

To define $\U{j}$ for $j\ge n_1+1$, we observe first that we may assume each family $\V{i}$ is a partition of $X$ itself, (cf. \cite[Corollary 8.3]{dydak-virk2016}). We then repeat the above procedure to arrive at
$\U{n_1+n_2}=\V{3}$. We repeat this entire process until we arrive at $\U{n_1+\cdots+n_{r-1}}=\V{r}$, which is uniformly bounded.
\end{proof}

A coarse space has sFCDC precisely when it satisfies condition (2) of Proposition~\ref{prop}. Condition (1) of Proposition~\ref{prop} is the coarse analog of countable asymptotic dimension. Since sFCDC is preserved by coarse direct products~\cite[Theorem 4.17]{bell-moran-nagorko2016}, we obtain the following. 

\begin{corollary}
The coarse version of Dydak's countable asymptotic dimension is preserved by coarse direct products.\qed
\end{corollary}

\bibliographystyle{abbrv}

\end{document}